\documentclass{amsart}[12pt]

\usepackage{amsmath, amsthm, amscd, amsfonts,}
\usepackage{graphicx,float}
\usepackage{amssymb}
\usepackage[cmtip,arrow]{xy}
\usepackage{pb-diagram,pb-xy}
\usepackage{comment}
\DeclareMathOperator{\dom}{dom}

\DeclareMathOperator{\Col}{Col}

\DeclareMathOperator{\GCH}{GCH}

\DeclareMathOperator{\Ult}{Ult}

\def\MPB{{\mathbb{P}}}

\def\a{\alpha}

\newtheorem{theorem}{Theorem}[section]
\newtheorem{lemma}[theorem]{Lemma}

\newtheorem{definition}[theorem]{Definition}

\newtheorem{claim}[theorem]{Claim}

\numberwithin{equation}{section}

\usepackage{pb-diagram}

\def\rmark{\mbox{$\rm\bf\rule{0.06em}{1.45ex}\kern-0.05em R$}}
\def\pmark{\mbox{$\rm\bf\rule{0.06em}{1.45ex}\kern-0.05em P$}}
\def\nmark{\mbox{$\rm\bf\rule{0.06em}{1.45ex}\kern-0.05em N$}}
\def\vdash{\mbox{$\rm\| \kern-0.13em -$}}

\def\rmark{\mbox{$\rm\bf\rule{0.06em}{1.45ex}\kern-0.05em R$}}
\def\pmark{\mbox{$\rm\bf\rule{0.06em}{1.45ex}\kern-0.05em P$}}
\def\nmark{\mbox{$\rm\bf\rule{0.06em}{1.45ex}\kern-0.05em N$}}
\def\vdash{\mbox{$\rm\| \kern-0.13em -$}}
\newcommand{\lusim}[1]{\smash{\underset{\raisebox{1.2pt}[0cm][0cm]{$\sim$}}
{{#1}}}}

\title[Changing measurable into small accessible cardinals]{Changing measurable into small accessible cardinals}

\author[M. Golshani]{Mohammad Golshani}

\thanks{The  author's research has been supported by a grant from IPM (No. 97030417).}

\begin{document}

%\subjclass[2010]{03E05, 03E35, 03E50, 03E55}

 \maketitle

\begin{abstract}
We give a detailed proof of the properties of the usual Prikry type forcing notion for turning a measurable cardinal into $\aleph_\omega$.
\end{abstract}

\section{introduction}
In this short note, we present a proof of the following known result.
\begin{theorem}
\label{main theorem}
Assume $\GCH$ holds and $\kappa$ is a measurable cardinal. Then there exists a generic extension in which $\kappa=\aleph_\omega$.
\end{theorem}
We try to give the details as much as possible to make it accessible to general audience who has some familiarity with forcing and large cardinals (see \cite{jech} for preliminaries).
\section{Proof of Theorem \ref{main theorem}}
Suppose that $\GCH$ holds and $\kappa$ is a measurable cardinal. Let $\mathcal U$ be a normal measure on $\kappa.$ Let also $j: V \to M \simeq \Ult(V, \mathcal)$
be the corresponding ultrapower embedding.
\begin{lemma}
\label{existence of guiding generics}
There exists $H \in V$ which a $\Col(\kappa^{++}, < j(\kappa))_M$-generic filter over $M$.
\end{lemma}
\begin{proof}
We have
\begin{enumerate}
\item $V \models$``$\Col(\kappa^{++}, < j(\kappa))_M$ is $\kappa^+$-closed''
\item $V \models$``$|\{A \in M: A$ is a maximal antichain in $\Col(\kappa^{++}, < j(\kappa))_M \}| \leq \kappa^+$''.
\end{enumerate}
Thus we can easily find the required $H$.
\end{proof}
We are now ready to define our main forcing construction.
\begin{definition}
A condition in $\MPB$ is of the form
\[
p= (\delta_0, f_0 \dots,  \delta_{n-1}, f_{n-1}, A, F)
\]
where
\begin{enumerate}
\item $n<\omega$.
\item $\delta_0 < \dots < \delta_{n-1}$.
\item For $i<n-1, f_i \in \Col(\delta_{i}^{++}, < \delta_{i+1})$.
\item $f_{n-1} \in \Col(\delta_{n-1}^{++}, < \kappa)$.
\item $A \in \mathcal U,$ and $\min(A) > \delta_{n-1}$.
\item $F$ is a function with $\dom(F)=A.$
\item For every $\delta \in A, F(\delta) \in \Col(\delta^{++}, < \kappa)$.
\item $[F]_{\mathcal U} \in H.$
\end{enumerate}
\end{definition}
Given a condition $p \in \MPB,$ we denote it by
\[
p= (\delta^p_0, f^p_0 \dots,  \delta^p_{n^p-1}, f^p_{n^p-1}, A^p, F^p).
\]
We also set
\begin{itemize}
\item $stem(p)=(\delta^p_0, f^p_0 \dots,  \delta^p_{n^p-1}, f^p_{n^p-1})$, the stem of $p$.
\item $u(p)=(A^p, F^p),$ the upper part of $p$.
\end{itemize}
\begin{definition}
Suppose $p, q \in \MPB.$
\begin{itemize}
\item [(a)] $p$ is an extension of $q$, $p \leq q$,  iff
\begin{enumerate}
\item $n^p \geq n^q$.
\item For all $i< n^q, \delta^p_i=\delta^q_i$.
\item For $n^q \leq i < n^p, \delta^p_i \in A^q$.
\item For $i < n^q, f^p_i \leq f^q_i$.
\item For $n^q \leq i < n^p, f^p_i \leq F^q(\delta^p_{i})$.
\item $A^p \subseteq A^q.$
\item For each $\delta \in A^p, F^p(\delta) \leq F^q(\delta)$.
\end{enumerate}
\item [(b)] $p$ is a direct extension of $q$, $p \leq^* q$, iff
\begin{enumerate}
\item $p \leq q.$
\item $n^p=n^q$.
\end{enumerate}
\end{itemize}
\end{definition}
We start by proving the basic properties of the forcing notion $\MPB$.
\begin{lemma}
\label{chain condition}
$(\MPB, \leq)$ satisfies the $\kappa^+$-c.c.
\end{lemma}
\begin{proof}
Let $A \subseteq \MPB$ be of size $\kappa^+$. Then, as
\[
\{stem(p): p \in A       \} \subseteq V_\kappa
\]
has size $\kappa$, we can find $p, q \in A$ such that $stem(p)=stem(q).$
We claim that $p$ an $q$ are compatible.
Since $[F^p]_{\mathcal U}, [F^q]_{\mathcal U} \in H,$ we can find $[F]_{\mathcal U} \in H$
such that $[F]_{\mathcal U} \leq [F^p]_{\mathcal U}, [F^q]_{\mathcal U}.$
Thus
\[
A^\ast=\{\delta < \kappa: F(\delta) \leq F^p(\delta), F^q(\delta)   \} \in \mathcal{U}.
\]
Set $A=A^\ast \cap A^p \cap A^q \in \mathcal{U}.$
Then
\[
stem(p)^{\frown} (A, F) \in \MPB
\]
and it extends both of $p$ and $q$.
\end{proof}
\begin{lemma}
\label{factorization}
Suppose $p \in \MPB$ and $m < n^p$. Then
\[
\MPB/p \simeq (\prod_{i<m} \Col((\delta^p_i)^{++}, <\delta^p_{i+1})) \times \MPB / p^{\geq m},
\]
where
$p^{\geq m} = (\delta^p_m, f^p_m, \dots, \delta^p_{n^p-1}, f^p_{n^p-1}, A^p, F^p)$. Further,
$(\MPB/ p^{\geq m}, \leq^*)$ is $\delta_{m}^{++}$-closed.
\end{lemma}
We now show that the forcing notion $(\MPB, \leq, \leq^*)$
satisfies the Prikry property.
\begin{lemma}
\label{prikry property}
Suppose $q \in \MPB$ and $\phi$ is a statement of the forcing language of $(\MPB, \leq)$. Then there exists $p \leq^* q$
which decides $\phi$.
\end{lemma}
\begin{proof}
We assume for simplicity that $n^q=0$ and $f^q_0=\emptyset.$
We write $q$ as $q=(A^q, F^q)$.
The proof has four main steps.
\begin{claim}
\label{reduction to stems}
(Reduction to stems) There exists $q^1 = (A^1, F^1) \leq^* q$ such that for any stem $s$,
\[
\exists s^{\frown}(A, F) \leq (A^1, F^1), ~ s^{\frown}(A, F) \parallel \phi \iff s^{\frown}(A^1, F^1) \parallel \phi.
\]
\end{claim}
\begin{proof}
For each stem $s$, if there exists $s^{\frown}(A, F) \leq (A^q, F^q)$
which decides $\phi,$ then let $(A^s, F^s)=(A, F)$
and otherwise set $(A^s, F^s)=(A^q, F^q)$.

Then $\{[F^s]_{\mathcal U}: s$ is a stem$ \} \subseteq H$,
and hence, we can find $[F^1]_{\mathcal U} \in H$ which extends all of them.
For each stem $s$, set
\[
B^s= \{\delta \in A^s: F^1(\delta) \leq F^s(\delta)           \} \in \mathcal U.
\]
Let also
\[
A^1 = \bigtriangleup_{s} B^s = \{\delta <\kappa: s \in V_\delta \Rightarrow \delta \in B^s                   \}.
\]
We show that $q^1 = (A^1, F^1)$ is as required. Thus suppose that $s$ is a stem and suppose there exists
$s^{\frown}(A, F) \leq (A^1, F^1)$ which decides $\phi$.

It then follows that $s^{\frown}(A^s, F^s)$ decides $\phi$. But, by our construction,
\[
s^{\frown}(A^1, F^1) \leq s^{\frown}(A^s, F^s)
\]
and hence $s^{\frown}(A^1, F^1) $ decides $\phi$ as well.
\end{proof}
Let $q^1=(A^1, F^1)$
be as in Claim \ref{reduction to stems}.
\begin{claim}
\label{reduction to stem minus top elemet}
(Reduction to stem minus top element)
There exists $q^2=(A^2, F^2) \leq^* q^1$ such that for any stem $s= (\delta_0, f_0, \dots, \delta_{n-1}, f_{n-1})$, if
$s^{\frown} (\delta_n, f_n)^{\frown} (A, F) \leq (A^2, F^2)$ and
\[
s^{\frown} (\delta_n, f_n) ^{\frown} (A, F) \parallel \phi,
\]
then
\[
s^{\frown} (\delta_n, F^2(\delta_n)) ^{\frown} (A^2, F^2) \parallel \phi.
\]
\end{claim}
\begin{proof}
Let $D$ be the set of all conditions $f \in \Col(\kappa^{++}, < j(\kappa))_M$ such that for any stem $s \in V_\kappa,$ if there exists $g \in \Col(\kappa^{++}, < j(\kappa))_M$ such that
$$s^{\frown} (\kappa, g)^{\frown} (j(A^1), j(F^1)) \parallel \phi,$$
then
$$s^{\frown} (\kappa, f)^{\frown} (j(A^1), j(F^1)) \parallel \phi.$$
We claim that $D \subseteq \Col(\kappa^{++}, < j(\kappa))_M$ is dense. Thus suppose that $g \in \Col(\kappa^{++}, < j(\kappa))_M.$
Let $(s_\alpha: \alpha < \kappa)$ enumerate all stems $s \in V_\kappa,$ and define a decreasing sequence $(f_\alpha: \alpha \leq \kappa)$ of
conditions in $\Col(\kappa^{++}, < j(\kappa))_M$, such that $f_0=g$ and for any $\alpha < \kappa,$
\[
\exists f \leq f_\alpha, s_\a^{\frown} (\kappa, f)^{\frown} (j(A^1), j(F^1)) \parallel \phi \implies s_\a^{\frown} (\kappa, f_{\alpha+1})^{\frown} (j(A^1), j(F^1)) \parallel \phi.
\]
Then $f=f_\kappa \in D$ and it extends $g.$

Let $f=[F^2]_{\mathcal U}$. We may assume that $[F^2]_{\mathcal U} \leq [F^1]_{\mathcal U}$, and hence
$A^\ast= \{ \delta < \kappa: F^2(\delta) \leq F^1(\delta)       \} \in \mathcal U.$

For any stem $s= (\delta_0, f_0, \dots, \delta_{n-1}, f_{n-1})$, we have $A^s \in \mathcal U,$ where
$A^s$ consists of those
$\delta_n \in A^1$ such that if there exists $s^{\frown} (\delta_n, f_n) ^{\frown} (A, F) \leq s^{\frown} (A^1, F^1)$ which decides $\phi$,
then
$$s^{\frown} (\delta_n, F^2(\delta_n)) ^{\frown} (A^2, F^2) \parallel \phi.$$
Let
$$A^2= \bigtriangleup_s A^s \cap A^\ast.$$

Then $q^2=(A^2, F^2)$ is easily seen to be as requested.
\end{proof}

\begin{claim}
\label{one point extension uniformization}
(One point extension uniformization)
There exists $q^3= (A^3, F^3) \leq^* q^2$ such that for any stem $s$, if
$s^{\frown} (\delta_n, f_n)^{\frown} (A, F) \leq (A^3, F^3)$ and
\[
s^{\frown} (\delta_n, f_n) ^{\frown} (A, F) \parallel \phi,
\]
then for all $\delta \in A^3,$
\[
s^{\frown} (\delta, F^3(\delta)) ^{\frown} (A^3, F^3) \parallel \phi.
\]
\end{claim}
\begin{proof}
Let $s$ be a stem. Set
\begin{itemize}
\item $A^s_0=\{\delta \in A^2:  s^{\frown} (\delta, F^2(\delta)) ^{\frown} (A^2, F^2) \Vdash ~\phi        \}$.
\item $A^s_1=\{\delta \in A^2:  s^{\frown} (\delta, F^2(\delta)) ^{\frown} (A^2, F^2) \Vdash ~\neg\phi        \}$.
\item $A^s_2=\{\delta \in A^2:  s^{\frown} (\delta, F^2(\delta)) ^{\frown} (A^2, F^2) \nparallel \phi        \}$.
\end{itemize}
Let $i_s < 3$ be such that $A^s =A^s_{i_s} \in \mathcal U.$ Let $A^3= \bigtriangleup_s A^s \in \mathcal U,$
and set $F^3 = F^2 \upharpoonright A^3.$

We show that $q^3=(A^3, F^3)$ is as required. Thus suppose that $s^{\frown} (\delta_n, f_n)^{\frown} (A, F) \leq (A^3, F^3)$ and
$s^{\frown} (\delta_n, f_n) ^{\frown} (A, F) \parallel \phi$. Let us suppose that it forces $\phi.$
Then $\delta_n \in A^s_0,$
and hence for all $\delta \in A_3$ such that $s^{\frown} (\delta, F^3(\delta)) ^{\frown} (A^3, F^3)$
is a condition, we have
$\delta \in A^s_0,$ and hence $s^{\frown} (\delta, F^2(\delta)) ^{\frown} (A^2, F^2) \Vdash ~\phi.$ It follows that
\[
s^{\frown} (\delta, F^3(\delta)) ^{\frown} (A^3, F^3) \Vdash ~\phi.
\]
\end{proof}
\begin{claim}
\label{minimal extension counterexample}
(Minimal extension counterexample)
There exists $q^4=(A^4, F^4) \leq^* q^3$ which decides $\phi.$
\end{claim}
\begin{proof}
Suppose not. Let $p \leq q^3 \leq $ decide $\phi,$ such that $n^p$ is minimal. By our assumption, $n^p>0$
and hence we can write $p$ as
\[
p=s^{\frown} (\delta_n, f_n) ^{\frown} (A, F).
\]
By Claim \ref{one point extension uniformization}
\[
\forall \delta \in A^3,~ s^{\frown} (\delta, F^3(\delta)) ^{\frown} (A^3, F^3) \Vdash \phi.
\]
Note that if $p^\ast \leq s^{\frown} (A^3, F^3) $ and $n^{p^\ast} > n^p-1$, then for some $\delta \in A^3, p^\ast \leq s^{\frown} (\delta, F^3(\delta)) ^{\frown} (A^3, F^3),$
and hence
\[
s^{\frown}(A^3, F^3) \Vdash \phi.
\]
This is in contradiction with the minimal choice of $n^p$.
\end{proof}
The lemma follows.
\end{proof}
Let $G$ be $\MPB$-generic over $V$. Let $C= (\delta_n: n< \omega)$ be the added Prikry sequence and for each $n<\omega$
set $G_n \subseteq \Col(\delta_n^{++}, <\delta_{n+1})$ be the generic filter added by $G$.

\begin{lemma}
\label{estimate of bounded sets}
Suppose $A \in V[G]$ and $A \subseteq \delta_n^{+}$. Then $A \in V[\prod_{i<n}G_i]$.
\end{lemma}
\begin{proof}
Let $p \in G$ be such that $n^p > n$. Let also $\lusim{A}$ be a $\MPB$-name for $A$ such that $\Vdash_{\MPB}$``$\lusim{A} \subseteq \delta_n^{+}$''.
Let $\lusim{B}$ be a $\MPB/ p^{\geq n}$-name for a subset of $\prod_{i<n} \Col(\delta_i^{++}, < \delta_{i+1}) \times \delta_n^+$ such that
\begin{center}
$ \Vdash_{\MPB/p^{\geq n}}$``$\forall \alpha < \delta_n^+ [~(q, \alpha) \in \lusim{B}) \iff q \Vdash_{\prod_{i<n} \Col(\delta_i^{++}, < \delta_{i+1})}$``$\alpha \in \lusim{A}$''~]''.
\end{center}

Let $(x_\alpha: \alpha \leq \delta_n^+)$ be an enumeration of $\prod_{i<n} \Col(\delta_i^{++}, < \delta_{i+1}) \times \delta_n^+$. Define a $\leq^*_{\MPB}$-decreasing sequence $(p_\alpha: \alpha \leq \delta_n^+ )$
of conditions in $\MPB / p^{\geq n}$ such that for each $\alpha< \delta_n^+, p_{\alpha+1} \|$``$x_\alpha \in \lusim{B}$''. This is possible as
$(\MPB/ p^{\geq n}, \leq^*)$ is  $\delta_n^{++}$-closed and satisfies the Prikry property. Then $p_{\delta_n^+}$ decides each ``$x_\alpha \in \lusim{B}$'', and so,
assuming $p_{\delta_n^+} \in G,$ we have
\begin{center}
$a= \{ \alpha < \kappa: \exists q \in \prod_{i<n}G_i, p_{\delta_n^+} \Vdash$``$(q, \alpha) \in \lusim{B}$''$ \} \in V[\prod_{i<n}G_i]$.
\end{center}
The result follows.
\end{proof}
The following is now immediate.
\begin{lemma}
\begin{enumerate}
Suppose $\lambda < \kappa$ is a cardinal.
\item [(a)] $V$ and $V[G]$ have the same bounded subsets of $\delta_0$.

\item [(b)] $Card^{V[G]} \cap [\delta_0, \kappa) = \bigcup_{n<\omega} \{ \delta_n, \delta_n^+, \delta_n^{++}\}$

\item [(c)] $V[G] \models$`` $\kappa=\delta_0^{+\omega}$''.
\end{enumerate}
\end{lemma}
Now forcing over $V[G]$ by $\Col(\aleph_0, < \delta_0),$
we get a model in which $\kappa$ becomes $\aleph_\omega.$ This completes the proof of Theorem \ref{main theorem}.

Mohammad Golshani,
School of Mathematics, Institute for Research in Fundamental Sciences (IPM), P.O. Box:
19395-5746, Tehran-Iran.

E-mail address: golshani.m@gmail.com

URL: http://math.ipm.ac.ir/golshani/

\end{document}